\documentclass[reqno]{amsart}

\usepackage{amsmath,amssymb,amsthm}
\usepackage[margin=3.5cm]{geometry}
\usepackage[colorlinks]{hyperref}
\usepackage{todonotes}

\newcommand{\incomp}{\mathbin{\|}}

\newcommand{\LL}{\mathcal{L}}
\newcommand{\MM}{\mathcal{M}}
\newcommand{\NN}{\mathcal{N}}
\newcommand{\ds}{\displaystyle}
\newcommand{\lp}{\left (}
\newcommand{\rp}{\right )}
\newcommand{\cG}{\mathcal{G}}

\DeclareMathOperator{\Inc}{Inc}
\DeclareMathOperator{\ldim}{ldim}
\DeclareMathOperator{\ldc}{ldc}

\DeclareMathOperator{\tdc}{tdc}

\DeclareMathOperator{\mult}{mult}

\DeclareMathOperator{\lbc}{lbc}

\def\epsilon{\varepsilon}

\theoremstyle{plain}
\newtheorem{thm}{Theorem}
\newtheorem{theorem}[thm]{Theorem}

\newtheorem{cor}[thm]{Corollary}

\newtheorem{lem}[thm]{Lemma}
\newtheorem{lemma}[thm]{Lemma}

\newtheorem{prop}[thm]{Proposition}

\theoremstyle{definition}
\newtheorem{conj}[thm]{Conjecture}

\newtheorem{question}[thm]{Question}
\newtheorem{defn}[thm]{Definition}

\theoremstyle{remark}
\newtheorem{rk}[thm]{Remark}

\renewcommand{\P}{\mathcal{P}}
\newcommand{\Q}{\mathcal{Q}}

\title[On difference graphs and the local dimension of posets]{On difference graphs and the local dimension of posets}

\author[]{Jinha Kim}
\address{Department of Mathematical Sciences, Seoul National University, Republic of Korea (Kim).}
\email{kjh1210@snu.ac.kr}

\author[]{Ryan R. Martin}
\address{Department of Mathematics, Ames, Iowa, USA (Martin).}
\email{rymartin@iastate.edu}

\author[]{Tom\'a\v{s} Masa\v{r}\'ik}
\address{Department of Applied Mathematics, Charles University, Czech Republic. (Masa\v{r}\'ik).}
\email{masarik@kam.mff.cuni.cz}

\author[]{Warren Shull}
\address{Department of Mathematics and Computer Science, Emory University, Atlanta, Georgia, USA (Shull).}
\email{warren.edward.shull@emory.edu}

\author[]{Heather C. Smith}
\address{School of Mathematics, Georgia Institute of Technology, Atlanta, Georgia, USA (Smith).}
\email{heather.smith@math.gatech.edu}

\author[]{Andrew Uzzell}
\address{Mathematics and Statistics Department, Grinnell College, Grinnell, Iowa, USA (Uzzell).}
\email{uzzellan@grinnell.edu}

\author[]{Zhiyu Wang}
\address{Department of Mathematics, University of South Carolina, Columbia, South Carolina, USA (Wang).}
\email{zhiyuw@math.sc.edu}

\begin{document}

\maketitle
\begin{abstract}
The dimension of a partially-ordered set (poset), introduced by Dushnik and Miller (1941), has been studied extensively in the literature. Recently, Ueckerdt (2016) proposed a variation called local dimension which makes use of partial linear extensions. While local dimension is bounded above by dimension, they can be arbitrarily far apart as the dimension of the standard example is $n$ while its local dimension is only $3$.

Hiraguchi (1955) proved that the maximum dimension of a poset of order $n$ is $n/2$. However, we find a very different result for local dimension, proving a bound of $\Theta(n/\log n)$. This follows from connections with covering graphs using difference graphs which are bipartite graphs whose vertices in a single class have nested neighborhoods.

We also prove that the local dimension of the $n$-dimensional Boolean lattice is $\Omega(n/\log n)$ and make progress toward resolving a version of the removable pair conjecture for local dimension.
\end{abstract}

\section{Introduction}

The order dimension (hereafter, dimension) of a poset, introduced by Dushnik and Miller~\cite{DM41} in 1941, has been studied extensively in the literature.
For a poset $\P=(P,\leq)$ with $x,y\in P$, we use the standard notation $x<y$ to indicate $x\leq y$ and $x\neq y$.
A \textit{realizer} of $\P$ is a non-empty family $\mathcal{L}$ of linear extensions of $\P$ so that $x<y$ in each $L\in \mathcal{L}$ if and only if $x <y$ in $\P$.
The \textit{dimension} of $\P$, denoted $\dim(\P)$, is the size of the smallest realizer.


We investigate a variant, called the local dimension, which was defined by Ueckerdt~\cite{U} and shared with the participants of the \textit{Order and Geometry Workshop} held in Gu\l{}towy, Poland, September 14-16, 2016. The definition was inspired by concepts studied in \cite{boxicity, KU16}.


\begin{defn}
    A \textit{partial linear extension} (abbreviated ``ple'') of a poset $\P$ is a linear extension of a subposet of $\P$.
    A \textit{local realizer} of $\P$ is a non-empty family $\mathcal{L}$ of ple's such that
    \begin{itemize}
        \item if $x< y$ in $\P$, then there is an $L\in \mathcal{L}$ with $x< y$ in $L$;
        \item if $x$ and $y$ are incomparable, then there are $L,L'\in \mathcal{L}$ with $x<y$ in $L$ and $x>y$ in $L'$.
    \end{itemize}
    Given a local realizer $\mathcal{L}$ of $\P$ and an element $x\in P$, the \emph{frequency} $\mu(x,\mathcal{L})$ is the number of ple's in $\mathcal{L}$ that contain $x$.
    The maximum frequency of a local realizer is denoted $\mu(\mathcal{L})=\max_{x\in P}\mu(x,\mathcal{L})$.
    The \emph{local dimension}, $\ldim(\P)$, of $\P$ is $\min_{\mathcal{L}} \mu(\mathcal{L})$ where the minimum is taken over all local realizers $\mathcal{L}$ of $\P$.
\end{defn}

Because each realizer $\mathcal{L}$ of $\P$ is also a local realizer where the frequency of each element is just $|\mathcal{L}|$,
\begin{align}
    \ldim(\P)\leq\dim(\P) . \label{eq:dimbound}
\end{align}

Hiraguchi~\cite{Hiraguchi} proved that the dimension of a poset $P$ with $n$ points is at most $\lfloor n/2\rfloor$ and the standard examples $S_n$ show that this is best possible. However, the local dimension of $S_n$ is only 3 for $n\geq 3$~\cite{U}. Our main result is the following bound for the local dimension of a poset with $n$ elements.
\begin{thm}\label{thm:size-main}
    The maximum local dimension of a poset on $n$ points is $\Theta(n/\log n)$.
\end{thm}

Our proof uses a correspondence between ple's and \emph{difference graphs} which are bipartite graphs in which there is an ordering of the vertices in one partite class such that their neighborhoods are nested. This relationship allows us to connect results about covering graphs with difference graphs to results about local dimension. 

A \emph{cover} of a graph $G$ is a set $\{H_i\}_{i\in [k]}$ of subgraphs of $G$ such that $\bigcup_{i\in [k]} E(H_i) = E(G)$. If all subgraphs in the cover are complete bipartite graphs, then we say $\{H_i\}_{i\in [k]}$ is a \emph{complete bipartite cover} of $G$. The \emph{local complete bipartite cover number} of $G$, denoted $\lbc(G)$, is the least $\ell$ such that there is a complete bipartite cover of $G$ in which every vertex of $G$ appears in at most $\ell$ of the subgraphs in the cover.

Since complete bipartite graphs are difference graphs, we make use of a theorem of Erd\H{o}s and Pyber \cite{Pyber} which states $\lbc(G) = O(n/\log n)$ for any graph $G$ with $n$ vertices to prove that the local dimension of any poset on $n$ points is $O(n/\log n)$. Their result is best possible, up to a constant factor, by the following theorem:

\begin{thm}[Chung, Erd\H{o}s, Spencer~\cite{CES}]\label{thm:CES}
    There is a graph $G$ such that for any cover of $E(G)$ with complete bipartite graphs, there is a vertex that appears in $\Omega(n/\log n)$ graphs in the cover. In other words, $\lbc(G) = \Omega(n/\log n)$.
\end{thm}

Because complete bipartite graphs are only a special type of difference graph, we use probabilistic tools to generalize this result of Chung, Erd\H{o}s, and Spencer to difference graphs.
Our Lemma~\ref{lem:diff_graph} is key to proving the lower bound for maximum local dimension in Theorem~\ref{thm:size-main}.

\begin{lem}\label{lem:diff_graph}
    There is a bipartite graph $G$ such that for any cover of $E(G)$ with difference graphs, there is a vertex that appears in $\Omega(n/\log n)$ graphs in the cover.
\end{lem}

This lemma may be of independent interest and the proof is given in Section~\ref{sec:diff_graph}.
The connection with local dimension is made in Section~\ref{sec:size}.

In Section~\ref{sec:Boolean}, we use the correspondence with difference graphs to give a counting argument for a lower bound for the local dimension of $\textbf{2}^{n}$, which denotes the subset lattice on $[n]$. The upper bound comes from \eqref{eq:dimbound} and the fact that $\dim\left(\textbf{2}^{n}\right)=n$.

\begin{thm}\label{thm:BooleanLatticeLB}
    For any positive integer $n$,
    \begin{align*}
        \frac{n}{2e\log n} \leq \ldim\left(\textbf{\rm\bf 2}^{n}\right) \leq n .
    \end{align*}
\end{thm}

The removable pair conjecture for dimension~\cite{Trotter}, which originated in 1971, states that for any poset $\P$ with at least 3 points, there is a pair of points $\{x,y\}$ such that $\dim(\P) \leq \dim(\P-\{x,y\})+1$.
The analogous conjecture (Conjecture~\ref{conj:TwoElementsRemoval}) can be made for local dimension.

\begin{conj}[Removable Pair]\label{conj:TwoElementsRemoval}
    For any poset $\P=(P,\le)$ for $|P|\ge3$, there are two elements $x,y$ in $P$ such that $\ldim(\P) \le \ldim(\P-\{x,y\})+1$.
\end{conj}

In Section~\ref{sec:rem_pair}, we extend a number of results about dimension to make partial progress toward resolving Conjecture~\ref{conj:TwoElementsRemoval}.
Using a classical result (Theorem~\ref{thm:bog}) by Bogart~\cite{Bog72} about the existence of a linear extension with certain properties, we prove that Conjecture~\ref{conj:TwoElementsRemoval} is true for posets of height two.

\begin{theorem}[Removable pair for posets of height two]\label{thm:twoElementsHeightTwo}
    For a poset $\P=(P, \le)$ with $|P|\ge3$ and height at most 2, there are two elements $x,y$ in $P$ such that
    \begin{align*}
        \ldim(\P) \le \ldim(\P-\{x,y\})+1 .
    \end{align*}
\end{theorem}

Furthermore we prove an analogous result to a theorem by Tator~\cite{Tator}, showing that one can find four elements whose removal decreases the local dimension by at most two.

\begin{theorem}[Removable quadruple]\label{thm:FourElementsRemoval}
   For a poset $\P=(P, \le)$ with $|P|\ge5$, there are four elements $x,y,z,w$ in $P$ such that
    \begin{align*}
        \ldim(\P) \le \ldim(\P-\{x,y,z,w\})+2 .
    \end{align*}
\end{theorem}

\section{Covering graphs with difference graphs}\label{sec:diff_graph}
In this section, we prove Lemma~\ref{lem:diff_graph} which is a result about graphs.
The connection to local dimension is made in Section~\ref{sec:size}.

First we define a class of graphs which is important for our proofs, known as \textit{difference graphs}.

\begin{defn}
    A \emph{difference graph} $H(a,b;f)$ is a bipartite graph on $a+b$ vertices with partite sets $U = \{u_1, \ldots, u_a\}$ and $W = \{w_1, \ldots, w_b\}$, equipped with a non-increasing function $f:[a] \to [b]$ such that $f(1)=b$ and, for all $i\in [a]$, $N(v_i) = \{w_1, \ldots, w_{f(i)}\}$ if $f(i)\geq 1$.
\end{defn}

\begin{rk}
    Difference graphs were first studied by Hammer, Peled, and Sun~\cite{HPS}.
    The definition of difference graphs used here, however, differs slightly in that we do not allow them to have isolated vertices, a convention that simplifies some of our proofs.
    In fact, we will use difference graphs to cover edges of a larger graph, so the change in definition is inconsequential.
\end{rk}
\begin{rk}
    The definition of $H$ above is symmetric with respect to the roles of $U$ and $W$.
    That is, if $H(a,b;f)$ is a difference graph, then the function $g(j):=\max\{i:f(i)\geq j\}$ witnesses that $H(b,a;g)=H(a,b;f)$.
\end{rk}

Let $\mathcal{H}_m$ be the collection of difference graphs with $m$ edges.
A \textit{partition} of an integer $m$ is a vector $p=(p_1,p_2,\ldots,p_k)$ such that $p_1\geq p_2\geq\cdots\geq p_k\geq 1$ and $p_1+p_2+\cdots+p_k=m$.
Let $\mathcal{P}_m$ be the collection of partitions of the integer $m$.
We claim that there is an injection $h:\mathcal{H}_m \to  \mathcal{P}_m$.
Indeed, given a difference graph $H(a,b;f)$ with $|E(H)|= \ds\sum_{i=1}^a f(i) = m$, define
\begin{align*}
    h(H(a,b;f)) = (f(i):i\in [a]) \in \mathcal{P}_m .
\end{align*}
In particular, $(f(1), f(2), \ldots, f(a))$ is a partition of $m$ into $a$ parts such that $f(1)=b$. Since $f$ is a non-increasing function, the partition $(f(1), f(2), \ldots, f(a))$ is unique to the choice of $a, b$ and $f$. It follows that $|\mathcal{H}_m| \leq |\mathcal{P}_m|$.

Hardy and Ramanujan~\cite{Hardy} and independently J.V.~Uspensky~\cite{Uspensky} gave the following  asymptotic formula for $|\mathcal{P}_m|$:
\begin{align} \label{eq:HRU}
    |\mathcal{P}_m| = \Theta\lp \frac{e^{c\sqrt{m}}}{m} \rp,
\end{align}
where $c = \pi  \sqrt{2/3}$.

\begin{defn}
    A \emph{difference graph cover} of a graph G is a family $\mathcal{H}$ of subgraphs of $G$ such that $E(G) = \bigcup_{H\in\mathcal{H}} E(H)$ and each $H$ is a difference graph.
    For a vertex $v\in G$, we use $\mult(v,\mathcal{H})$ to denote the number of difference graphs in $\mathcal{H}$ that contain $v$.
    The \emph{total difference graph cover number} is defined by
    \[ \tdc(G) = \ds\min \left\{\ds\sum_{H\in \mathcal{H}} |V(H)|: \mathcal{H} \text{ is a difference graph cover of } G \right \}.\]
The \emph{local difference graph cover number} of $G$, denoted by $\ldc(G)$, is defined as
\begin{align*}
    \ldc(G)     =     \ds\min \left\{\ds\max_{v\in V(G)} \left\{ \mult(v,\mathcal{H}) \right\} :\mathcal{H} \text{ is a difference graph cover of } G\right\}.
\end{align*}
\end{defn}

Let $G=\cG(n_1,n_2,p)$ be a {\em random bipartite graph} with partite sets $V_1$ and $V_2$, of order $n_1$ and $n_2$ respectively, in which each pair $\{i,j\} \in V_1 \times V_2$ appears independently as an edge in $G$ with probability $p$.
We say an event in a probability space holds asymptotically almost surely (a.a.s.) if the probability that it holds tends to $1$ as $n$ goes to infinity.

Proposition~\ref{prop:diff_graph} guarantees a graph $G$ such that $\tdc(G)$ is large and Corollary~\ref{cor:sharp-ex} demonstrates that $\ldc(G)$ is also large, which establishes Lemma~\ref{lem:diff_graph}.

For simplicity, we will assume that $n$ is even and note that a similar bound is attained for odd $n$ by simply adding a single isolated vertex to the prior even case.
\begin{prop}\label{prop:diff_graph}
    Let $\epsilon>0$ and let $n$ be a sufficiently large even integer. There exists a bipartite graph~$G = (A \cup B, E)$ with partite sets satisfying $|A|=|B| = n/2$ such that $$\tdc(G) \geq \biggl(\frac{1-2\epsilon}{4e} \biggr)\frac{n^2}{\ln{n}}.$$
\end{prop}

\begin{proof}
Fix $\epsilon\in(0,1/2)$ and let $G \sim \cG(n/2, n/2,1/e)$ be a random bipartite graph with partite sets $A$ and $B$.
We will prove a stronger statement: $\tdc(G) \geq \left(\frac{1-2\epsilon}{4e} \right)\frac{n^2}{\ln{n}}$ a.a.s.

If a subgraph of $G$ is isomorphic to a difference graph $H(a,b;f)$ which has partite sets $U$ and $W$, then say $H(a,b;f)$ is a subgraph of $G$ (and write $H(a,b;f)\subseteq G$). By the symmetry of difference graphs, we may assume $U\subseteq A$ and $W\subseteq B$.

Each difference graph $H=H(a,b;f)$ which is a subgraph of $G$ is one of two types: $H$ is \emph{type I} if
\[\frac{|V(H)|}{|E(H)|} < \frac{1-\epsilon}{\ln{n}}.\]
Otherwise $H$ is \emph{type II}. \\

\noindent {\bf Claim:} $G$ contains no type I difference graphs as subgraphs, a.a.s.

Indeed, for fixed $(a, b; f)$ with $H=H(a,b;f)$ type I, the probability that $H$ is a subgraph of $G$ is at most $e^{-\epsilon |E(H)|}$:
\begin{align*}
    \Pr \lp H(a,b;f) \subseteq G \rp
        &\leq \lp\frac n2\rp^a\lp\frac n2\rp^b\lp\frac{1}{e}\rp^{|E(H)|}\\
        &\leq \exp\left((a+b)\ln{n}-|E(H)|\right)\\
        &< \exp\left(\frac{1-\epsilon}{\ln{n}}  |E(H)| \ln{n} -|E(H)|\right) \\ 
        &= \exp\left(-\epsilon |E(H)|\right).
\end{align*}
For each type I difference graph $H\subseteq G$, we have the following bounds on $|E(H)|$:
\begin{align*}
    2\ln n \leq \frac{2\ln{n}}{1-\epsilon} \leq |V(H)|\frac{\ln{n}}{1-\epsilon} < |E(H)| \leq \left(\frac{n}{2}\right)^2 .
\end{align*}
Let $T$ be the event that $G$ contains a type I difference graph as a subgraph. By the relationship between $\mathcal{H}_m$ (the number of difference graphs with $m$ edges) and $\mathcal{P}_m$ (the number of integer partitions of $m$), we can use \eqref{eq:HRU} to obtain the following bound when $n$ is sufficiently large:
\begin{align*}
    \Pr \lp T \rp
        \leq \sum_{m = 2\ln n}^{n^2/4} |\mathcal{H}_m| e^{-\epsilon m}
        \leq \sum_{m = 2\ln n}^{n^2/4} |\mathcal{P}_m| e^{-\epsilon m}
        = O\lp \sum_{m = 2\ln n}^{n^2/4} \frac{e^{c\sqrt{m}}}{m} e^{-\epsilon m} \rp
        = o(1).
\end{align*}
As a result, all difference graphs which are subgraphs of $G$ are type II, a.a.s. This completes the proof of the claim. \\

Now since $G \sim \cG(n/2,n/2,1/e)$ is a random bipartite graph, it follows that
\begin{align*}
    \mathbb{E}(|E(G)|) = \frac{n^2}{4e} .
\end{align*}

Applying a Chernoff bound (\cite[Theorem A.1.13]{Alon-Spencer}), we have that
\begin{align*}
    \Pr\lp |E(G)|
        <    \frac{1-2\epsilon}{1-\epsilon}\mathbb{E}(|E(G)|) \rp
        \leq \exp\lp  -\frac{\epsilon^2}{2(1-\epsilon)^2}\mathbb{E}|E(G)| \rp
        = \exp \lp - \frac{\epsilon^2n^2}{8e(1-\epsilon)^2} \rp
            = o(1).
\end{align*}
It follows that $|E(G)| \geq \frac{(1-2\epsilon)n^2}{(1-\epsilon)4e}$ a.a.s.

Thus a.a.s.~$G$ contains no Type I difference graphs and has at least $\frac{(1-2\epsilon)n^2}{(1-\epsilon)4e}$ edges.
As a result, for any difference graph cover $\{H_1, H_2,\cdots, H_\ell\}$ of G that witnesses $\tdc(G)$, we have that
\begin{align*}
    \tdc(G)
        = \ds\sum_{i=1}^\ell |V(H_i)|
        \geq \ds\sum_{i=1}^\ell |E(H_i)| \frac{1-\epsilon}{\ln{n}}
        \geq \lp \frac{1-\epsilon}{\ln{n}} \rp |E(G)|
        \geq \frac{1-2\epsilon}{4e} \lp\frac{n^2}{\ln{n}}\rp.
\end{align*}
\end{proof}

\begin{cor}\label{cor:sharp-ex}
Let $\epsilon>0$ and let $n$ be sufficiently large. There exists a bipartite graph $G$ satisfying
\begin{align*}
    \ldc(G) \geq \frac{\tdc(G)}{n} \geq \lp\frac{1-2\epsilon}{4e}\rp\frac{n}{\ln{n}}.
\end{align*}
\end{cor}

\section{Bounding local dimension by size}\label{sec:size}

Before we prove Theorem \ref{thm:size-main}, we need some definitions. Let $\P=(P,\le)$ be a poset with $n$ elements.
To each element $x \in P$, we associate $x$ with two new elements $x'$ and $x''$.
The \textit{split} of $\P$ (defined by Kimble~\cite{Kimble}) is a height-two poset $\Q$ with minimal elements $\{x': x\in P\}$ and maximal elements $\{x'':x\in P\}$ such that for all $x,y\in P$, $x'\leq y''$ in $\Q$ if and only if $x \leq y$ in $\P$.
The following lemma relates the local dimension of $\P$ and $\Q$.
\begin{lem}[Barrera-Cruz, Prag, Smith, Taylor, Trotter~\cite{Barrera-Cruz}]\label{lemma:split}
    If $\Q$ is the split of a poset $\P$, then
    \begin{align*}
        \ldim(\Q)-2\leq \ldim(\P) \leq 2\ldim(\Q)-1 .
    \end{align*}
\end{lem}

Let's also recall a classical theorem on partitioning the edges of a graph into complete bipartite graphs.
\begin{thm}[Erd\H{o}s, Pyber~\cite{Pyber}]\label{thm:Erdos-Pyber}
    Let $G = (V,E)$ be a graph on $n$ vertices.
    The edge set $E$ can be partitioned into complete bipartite graphs such that each vertex $v \in V$ is contained in $O\lp n/\log n\rp$ of the bipartite subgraphs.
\end{thm}

Csirmaz, Ligeti, and Tardos~\cite{CsLT} showed that such a partition can be achieved so that each vertex is in at most $(1+o(1))\frac{n}{\log_2 n}$ of the bipartite subgraphs.

We are now ready to prove Theorem \ref{thm:size-main}. Let's start with the upper bound.
\begin{lemma}\label{lemma:upper-bound}
    For any poset $\P$ with $n$ points, $\ldim(\P) \leq (1+o(1))\frac{4n}{\log_2 (2n)}$.
\end{lemma}

\begin{proof}
Let $\Q$ be the split of $\P$.
Suppose $\Q$ has minimal elements $A=\{a_i: i\in [n]\}$ and maximal elements $B=\{b_i: i\in [n]\}$.
By Lemma~\ref{lemma:split}, it suffices to show that $\ldim(\Q) \leq (1+o(1))\frac{2n}{\log_2 (2n)}$.

We will exhibit a local realizer for $\Q$ such that each element is contained in at most $(1+o(1))\frac{2n}{\log_2 (2n)}$ ple's.
Begin with two linear extensions $L_1$ and $L_2$, each with block structure $A<B$ and, for any pair $i,j\in [n]$, $a_i<a_j$ in $L_1$ if and only if $a_i>a_j$ in $L_2$ and similarly for the elements of $B$.
It remains to construct a set $\mathcal{M}$ of partial linear extensions for $\Q$ such that $a_i>b_j$ in some $M \in \mathcal{M}$ precisely when $a_i$ and $b_j$ are incomparable in $\P$.

Construct an auxiliary bipartite graph $G = (A\cup B,E)$ where $ab \in E(G)$ if and only if $a\in A$ and $b\in B$ are incomparable in $\Q$.
Now by Theorem~\ref{thm:Erdos-Pyber} (or, precisely~\cite{CsLT}), $E$ can be partitioned into complete bipartite graphs $G_1, \ldots, G_m$ such that each vertex $v \in V$ is contained in at most $(1+o(1))\frac{(2n)}{\log_2 (2n)}$ of the bipartite subgraphs because $G$ has $2n$ vertices.
Each $G_i$ corresponds to a ple of $\Q$ as follows: Suppose that $V(G_i) = A_i \cup B_i$ with $A_i\subseteq A$ and $B_i\subseteq B$.
Then let $M_i$ be a ple of $\Q$ on the ground set $V(G_i)$ with block structure $B_i<A_i$.
Since $G_i$ is a complete bipartite subgraph of $G$, it follows that for all $a \in A_i$ and $b\in B_i$, $a$ and $b$ are incomparable in $\Q$.
Thus $M_i$ is indeed a ple of $\Q$.

So $\LL = \{L_1, L_2, M_1, M_2, \ldots, M_m\}$ is a local realizer of $\Q$ in which every element of $Q$ appears at most $(1+o(1))\frac{2n}{\log_2 (2n)}$ times in $\mathcal{L}$ as desired.
\end{proof}

To show that the bound in Lemma~\ref{lemma:upper-bound} is best possible to within a multiplicative constant, we describe a connection between difference graphs and partial linear extensions of height-two posets.

For a height-two poset $\P$ with minimal elements $A$ and maximal elements $B$, a \emph{critical pair} is an incomparable pair $(a,b)\in A \times B$.
Define $G=G(\P)$ to be a bipartite graph with partite classes $A$ and $B$ such that $ab \in E(G)$ if and only if $(a,b)$ is a critical pair for $\P$.

Consider a ple $L$ with block structure $B_1<A_2<B_2<\ldots<A_m$ (for some $m\in \mathbb{N}$) where $A_i\subseteq A$ and $B_i\subseteq B$ for each $i\in [m]$.
Let $H(L)$ be the subgraph of $G$ with vertices $\bigcup_{i\in [m]} (A_i \cup B_i)$ and edges $\{ab: a>b \text{ in } L\}$.
Since for each $a,a'\in A$, either $N_{H(L)}(a) \subseteq N_{H(L)}(a')$ or $N_{H(L)}(a) \supseteq N_{H(L)}(a')$, the subgraph $H$ is a difference graph.


\begin{lemma}\label{lemma:sharpness}
    There exists a poset $\P$ with $n$ points satisfying $\ldim(\P) = \Omega \left(n/\log n\right)$.
\end{lemma}
\begin{proof}
We may assume that $n$ is even.
If $n$ is odd, then construct $\P$ as below on $n-1$ elements and add a point incomparable to everything else. The local dimension will increase by at most 1.

Let $G$ be the bipartite graph guaranteed by Corollary \ref{cor:sharp-ex} with partite classes $A$ and $B$ where $|A|=|B| = n/2$ and $\ldc(G) = \Omega(n/\log n)$.
Construct a height-two poset $\P$ where $A$ and $B$ are the minimal and maximal elements respectively, and $a \leq b$ in $\P$ if and only if $a\in A$, $b\in B$, and $ab \notin E(G)$.

Let $\mathcal{M}$ be an arbitrary local realizer of $\P$ with $\mu(\mathcal{M}) = \ldim(\P)$. We will create a different local realizer $\mathcal{M}'$ with $\mu(\mathcal{M}') \leq \mu(\mathcal{M})+2$ and with the property that $\mu(\mathcal{M'}) = \Omega(n/\log n)$. This will prove the lemma.

Let $L$ and $L'$ be two linear extensions of $\P$, each with block structure $A<B$ such that for any pair $a,a'\in A$, $a<a'$ in $L$ if and only if $a'<a$ in $L'$ and similarly for $B$.

Each $M\in \mathcal{M}$ has block form $A_1<B_1<A_2<B_2<\ldots<A_t<B_t$ for some $t\in \mathbb{N}$ where $A_1$ and $B_t$ may be empty. Create a new ple, $M'$, from $M$ simply by deleting all elements in $A_1$ and $B_t$.

So $\mathcal{M}'= \{L,L'\} \cup \{M': M\in \mathcal{M}\}$ is another local realizer of $\P$ with $\mu(\mathcal{M}') \leq \mu(\mathcal{M})+2$.
To see this, observe that every pair of elements of $A$ and each pair of elements in $B$ are reversed by the linear extensions $L$ and $L'$.
Every comparable pair is realized in $L$. Moreover, for each critical pair $(a,b) \in A\times B$ of $\P$, we have $a<b$ in $L$ and there is a ple $M \in \mathcal{M}$ with $a>b$.
Hence $a>b$ in the corresponding $M' \in \mathcal{M}'$ also.

So the difference graphs that correspond to the ple's in $\{M': M \in \mathcal{M}\}$ form a difference graph cover of $G$.
Since $\ldc(G) = \Omega \left(n/\log n\right)$, it follows that $\mu(\mathcal{M'}) = \Omega\left(n/\log n\right)$. Since $\mathcal{M}$ was a local realizer with $\mu(\mathcal{M}') \leq \mu(\mathcal{M})+2=\ldim(\P)+2$,
we have proved that $\ldim(\P) = \Omega \left(n/\log n\right)$.
\end{proof}

Theorem~\ref{thm:size-main} follows immediately from Lemma~\ref{lemma:upper-bound} and Lemma~\ref{lemma:sharpness}.

\section{Cartesian products and the Boolean lattice}\label{sec:Boolean}



In this section, we explore the local dimension of products of posets and the Boolean lattice.

\begin{defn}
    For two posets $\P=(P,\leq_{\P})$ and $\Q=(Q,\leq_{\Q})$, the \emph{Cartesian product} of $\P$ and $\Q$ is the poset $\P\times\Q=(P\times Q, \leq_{\P\times\Q})$, where $(p_1,q_1) \leq_{\P\times\Q} (p_2,q_2)$ if and only if $p_1 \leq_{\P} p_2$ and $q_1 \leq_{\Q} q_2$ in $Q$.
\end{defn}

\begin{thm}\label{thm:product}
    For any two posets $\P$ and $\Q$, $\ldim(\P\times\Q) \leq \ldim(\P)+\ldim(\Q)$.
\end{thm}

\begin{proof}
Let $\LL = \{L_1, \ldots, L_s\}$ and $\MM = \{M_1, \ldots, M_t\}$ be local realizers of $\P$ and $\Q$ respectively, such that $\mu(\LL) =\ldim(\P)$ and $\mu(\MM)=\ldim(\Q)$.


Let $L_0$ be fixed a linear extension of~$\P$ and let $M_0$ be a fixed linear extension of~$\Q$.
For each $i\in [s]$, define a ple ${L'_i}$ on $\P\times\Q$  with elements $\{(a,b): a\in L_i, b\in Q\}$ such that $(a,b)<(a',b')$ in $L'_i$ if and only if (1) $a < a'$ in $L_i$ or (2) $a=a'$ and $b<b'$ in $M_0$.
For each $j\in [t]$, we define $M'_j$ similarly.
Let $\NN = \{{L'_1}, \dots, {L'_s}\} \cup  \{{M'_1}, \dots, {M'_t}\}$. We claim that $\mathcal{N}$ is a local realizer for $\P\times \Q$ with $\mu(\NN) \leq \ldim(\P) + \ldim(\Q)$.

Observe that $(x,y) \in P \times Q$ appears in ${L'_i}$ if and only if $x \in L_i$ and appears in ${M'_j}$ if and only if $y \in M_j$.
Thus,
\begin{align*}
    \mu\bigl((x,y), \NN \bigr) = \mu(x, \LL) + \mu(y, \MM) \leq \ldim(\P) + \ldim(\Q).
\end{align*}


To see that $\NN$ is a local realizer of~$\P\times\Q$, consider two pairs $(a,b),(c,d)\in P\times Q$.
If $(a,b) \leq (c,d)$ in $\P\times\Q$, then $a \leq c$ in $\P$.
Because $\LL$ is a local realizer of~$\P$, there exists $i$ such that $a \leq c$ in $L_i$.
By the definition of~${L'_i}$, $(a,b) \leq (c,d)$ in ${L'_i}$.

If $(a,b) \incomp (c,d)$, it suffices to prove that there is a ple in $\NN$ with $(a,b)>(c,d)$. When $a=c$ or $b=d$, the result follows easily from the fact that $\LL$ and $\MM$ are local realizers for $\P$ and $\Q$, so we assume $a\neq c$ and $b\neq d$.
Since $(a,b) \incomp (c,d)$, one of the following holds: (1) $a \incomp c$, (2) $a>c$ while $b < d$, (3) $b\incomp d$, or (4) $a<c$ with $b>d$.
For cases (1) and (2), we have $a \incomp c$ or $a>c$.
Because $\LL$ is a local realizer of~$\P$, there exists $L_i\in \LL$ with $a>c$. Therefore $(a,b) > (c,d)$ in $L'_i$. The argument is similar if $b\incomp d$ or $b>d$.

Thus, $\NN$ is a local realizer of~$\P \times \Q$.
\end{proof}



Now consider the Boolean lattice $\textbf{2}^{n}$ which is the Cartesian product of $n$ chains of height 2.
According to Theorem~\ref{thm:product}, the local dimension of $\textbf{2}^{n}$ is at most $n$ because the local dimension of a chain is 1.

For any integer $s\in\{0,1,\ldots,n\}$, we denote $\binom{[n]}{s}$ to be all the subsets of $[n]$ that have size equal to $s$.
We call this ``layer $s$'' or, when not a tongue-twister, the $s^{\rm th}$ layer. Let $P(s,t;n)$ be the subposet of $\textbf{2}^{n}$ induced by layers $s$ and $t$.

Following the notation in~\cite{BKKT}, we let $\dim(s,t;n)$ and $\ldim(s,t;n)$ denote, respectively, the dimension and the local dimension of $P(s,t;n)$.
Since both dimension and local dimension are monotone under the deletion of elements,  $\ldim(s,t;n)$ gives a lower bound on $\ldim(\textbf{2}^{n})$.

Note that $\dim(1,n-1;n)=n$ because those layers form a standard example, but $\ldim(1,n-1;n)=3$. Hurlbert, Kostochka, and Talysheva~\cite{HKT}, established that $\dim(2,n-2;n)=n-1$ if $n\geq 5$ and $\dim(2,n-3;n)=n-2$ for $n\geq 6$. Moreover, F\"uredi~\cite{Fur} showed that for every $k\geq 3$ and $n$ sufficiently large, $\dim(k,n-k;n)=n-2$.

In order to establish a lower bound of $(1-o(1))\frac{n}{2e\ln n}$ in Theorem~\ref{thm:BooleanLatticeLB}, we again use difference graphs.

\begin{proof}[Proof of Theorem~\ref{thm:BooleanLatticeLB}]


For $n$ sufficiently large and $k=\lceil n/e \rceil$, we will show that $\ldim(1,n-k;n) = \Omega(n/\log n)$.
Consider the auxilary bipartite graph $G=G(1,n-k;n)=(\mathcal{V},\mathcal{S};E)$ with a vertex in $\mathcal{V}$ for each singleton, a vertex in $\mathcal{S}$ for each set of size $n-k$, and $\{i\}\in\mathcal{V}$ is adjacent to $S\in \mathcal{S}$ if and only if $i\not\in S$.
In other words, two vertices are adjacent if they represent a critical pair in $P(1,n-k;n)$.


Let $b=2\ln n$.
As we have seen above, the local dimension of $P(1,n-k;n)$ is at least $\ldc(G)$. For a difference graph $H$ which is a subgraph of $G$ and a set $S\in \mathcal{S}$, we say $H$ is \textit{small in $S$} if there are less than $b$ edges incident to $S$ in $H$.
Otherwise, we say that $H$ is \textit{big in $S$}.
Any difference graph that is big in some $S$ is said to be \textit{big} itself.

 Let $\mathcal{H}$ be a difference graph cover of $G$ that realizes $\ell:=\ldc(G).$ We will consider two cases.
 
First suppose there is a set $S\in \mathcal{S}$ such that no $H\in \mathcal{H}$ is big in $S$, then all $k$ edges incident with $S$ must be covered by difference graphs that each contain at most $b-1$ of them. So $S$ appears in at least $\frac{k}{b-1} \geq \frac{k}{b} \geq \frac{n}{2e\ln n}$ difference graphs in $\mathcal{H}$ and as a result $\ldc(G) \geq  \frac{n}{2e\ln n}$ as desired. 
 
Now suppose that for each set $S\in \mathcal{S}$, there is at least one $H\in \mathcal{H}$ that is big in $S$. Recall that the neighborhoods of sets in $\mathcal{S}$ are nested in $H$, so if $H$ is big in $S_1, S_2, \ldots, S_t$, then there are $b$ singletons that are adjacent to each of these sets in $H$. In particular, these $b$ singletons are not elements of any of $S_1, S_2, \ldots, S_t$. So $t \leq \binom{n-b}{n-k}= \binom{n-b}{k-b}$.

Since there are at most $\ell$ difference graphs containing any one singleton and each big difference graph contains at least $b$ singletons, there are at most $\ell n/b$ big difference graphs.
Hence, there are at most $\frac{\ell n}{b}\binom{n-b}{k-b}$ sets $S$, with multiplicity, for which there is a difference graph in $\mathcal{H}$ which is big in $S$. Since for every $S\in\mathcal{S}$ there is a difference graph in $\mathcal{H}$ that is big in $S$, we have the following inequality:
\begin{align}
    \frac{\ell n}{b}\binom{n-b}{k-b} \geq \binom{n}{k}. \label{subset:ub}
\end{align}
Since $k=\lceil n/e \rceil$ and $b=2\ln n$, in this case we have
\begin{align*}
    \ell    &\geq    \frac{b}{n}\binom{n}{k}\binom{n-b}{k-b}^{-1}
            \geq    \frac{b}{n}\left(\frac{n}{k}\right)^b   
                    \geq     \frac{\ln n}{n}e^{2\ln n} = n \ln n.
\end{align*}
Therefore, for $n$ sufficiently large, we may conclude
    $\ldim(\textbf{2}^n) \geq  \ldim(1,n-k;n)   \geq    \frac{n}{2e\ln n}.$
\end{proof}

\section{Removable pair and quadruple}\label{sec:rem_pair}

In this section, we consider the analogue of the  removable pair conjecture for local dimension.
Recall the following theorem:
\begin{theorem}[Bogart~\cite{Bog72}]\label{thm:bog}
    Let $\P$ be poset and let $\Inc(\P) = \{(x,y): x,y\in P \text{ and } x\incomp y\}$.
    Suppose $C_a$ and $C_b$ are chains of~$P$ such that $(a,b)\in \Inc(\P)$ for each $x\in C_a$ and $y\in C_b$. Then there is a linear extension $L$ of $\P$ with $x<y$ in $L$ for each $(x,y)\in \Inc(\P)$ with $x\in C_a$ or $y\in C_b$.
\end{theorem}

We use this result to show that, by removing any two chains from a poset in which no element in the first chain relates to any in the second, the local dimension decreases by at most two.

\begin{theorem}[Two Chain Removal]\label{thm:twoChainsRemoval}
    If $C_1$ and $C_2$ are chains of the poset $\P$ with $\P-(C_1 \cup C_2)$ nonempty and each element of $C_1$ is incomparable with each element of $C_2$, then
    \begin{align*}
        \ldim(\P)     \le     \ldim(\P-(C_1 \cup C_2))+2 .
    \end{align*}
\end{theorem}

\begin{proof}
Take a local realizer $\LL$ of $\P-(C_1 \cup C_2)$.
Let $L_1$ be the linear extension obtained from Theorem~\ref{thm:bog} when $C_1 = C_a$ and $C_2=C_b$. Reversing the roles of $C_1$ and $C_2$, let $L_2$ be the linear extension from Theorem~\ref{thm:bog} when $C_1 = C_b$ and $C_2=C_a$.
Then $\LL'=\LL \cup \{L_1,L_2\}$ is a local realizer of $\P$ since  all the incomparabilities between $C_1$ and $C_2$ were reversed, while preserving comparabilities because $L_1$ and $L_2$ are linear extensions.
Since $\mu(\LL')=\mu(\LL)+2$, we obtain $\ldim(\P) \le \ldim(\P-(C_1 \cup C_2))+2$.
\end{proof}

When one of the chains in Theorem~\ref{thm:twoChainsRemoval} is empty, we obtain the following corollary.
\begin{cor}\label{cor:oneChainRemoval}
If $C$ is a chain of the poset $\P$ and $\P-C$ is nonempty, then
\begin{align*}
    \ldim(\P) \le \ldim(\P-C)+2.
\end{align*}
\end{cor}

If a poset has incomparable elements, one minimal and one maximal, then deleting them decreases the local dimension by at most 1.

\begin{lemma}\label{lem:twoElementsMaximal}
    If $\P$ is a poset with at least 3 elements, such that $x$ is a minimal element, $y$ is a maximal element, and $x$ is incomparable to $y$, then
    \begin{align*}
        \ldim(\P)     \le     \ldim(\P-\{x,y\})+1 .
    \end{align*}
\end{lemma}

\begin{proof}
Let $\LL=\{L_1,\dots,L_t\}$ be a local realizer of $\P-\{x,y\}$ such that $\mu(\LL)=\ldim(\P-\{x,y\})$.
Let $z$ be an arbitrary element of $\P-\{x,y\}$.
By relabeling, we may assume that $\{L_1,\ldots,L_d\}$ is precisely the set of all $L_i \in \LL$ where $z \in L_i$.
Note each element in $\P-\{x,y\}$ appears in at least one of the ple's in $\{L_1, \ldots, L_d\}$.

For each $i\in \{1,\dots,d\}$, create a new ple $M_i$ with block structure $x<L_i<y$ such that if $\ell< \ell'$ in $L_i$, then  $\ell<\ell'$ in $M_i$.

Let $L'$ be the linear extension of $\P$ guaranteed by Theorem~\ref{thm:bog} when $C_a= \{x\}$ and $C_b=\{y\}$.
Then $\LL'    =    \{M_1,\dots,M_d\}\cup\{L_{d+1},\dots,L_d\}\cup\{L'\}$ is a local realizer of $\P$. Further, $\mu(\LL') = \mu(\LL) +1 = \ldim(\P-\{x,y\})+1$ as desired.
\end{proof}

Further, if a poset has one minimal element and one maximal element such that each element is related to at least one of them, then deleting both decreases the local dimension by at most 1.

\begin{lemma}\label{lem:twoElementsSpecial}
    If $\P=(P,\le)$ be a poset with $|P|\ge3$, such that $x\in P$ is a minimal element, $y\in P$ is a maximal element, $x<y$, and there is no element that is incomparable to both $x$ and $y$, then
    \begin{align*}
        \ldim(\P)     \le     \ldim(\P-\{x,y\})+1 .
    \end{align*}
\end{lemma}

\begin{proof}
Let $\LL=\{L_1,\dots,L_t\}$ be a local realizer of $\P-\{x,y\}$ such that $\mu(\LL)=\ldim(\P)$.
Let $z$ be an arbitrary element of $P-\{x,y\}$. As before, we may assume that $\{L_1,\ldots,L_d\}$ is precisely the set of all ple's in $\LL$ that contain $z$. Further, every element of $\P-\{x,y\}$ appears in at least one of these ple's.

For each $i\in\{1,\dots,d\}$, modify $L_i$ by adding $x$ and $y$ to obtain $M_i$ with block structure $x < L_i < y$.
Let $I_x$ be the set of all elements incomparable to $x$ and let $I_y$ the set of all elements incomparable to $y$. Let $R_x$ be a ple with block structure $I_x<x$ and $R_y$ a ple with block structure $y<I_y$.
Then $
    \LL'     =     \{M_1,\dots,M_d\}\cup\{L_{d+1},\dots,L_t\}\cup\{R_x,R_y\}$
is a local realizer of $\P$. Since $I_x \cap I_y=\emptyset$, one can quickly see that $\mu(\LL') = \mu(\LL)+1  = \ldim(\P-\{x,y\}) + 1$.
\end{proof}

The previous two lemmas can now be used to prove the Removable Pair Conjecture in the case of height 2 posets.

\begin{proof}[Proof of Theorem~\ref{thm:twoElementsHeightTwo}]
Let $\P$ be a height-two poset with minimal elements $A$ and maximal elements $B$.
If $\P$ has two elements $x\in A$ and $y\in B$ with $x$  incomparable to $y$, then $\ldim(\P) \le \ldim(\P-\{x,y\})+1$ by Lemma~\ref{lem:twoElementsMaximal}.

If however every element of $A$ is comparable to every element of $B$, then pick a pair $(x,y)\in A\times B$. Since $a<y$ for each $a\in A$ and $x<b$ for each $b\in B$, then by Lemma~\ref{lem:twoElementsSpecial}, $\ldim(\P) \le \ldim(\P-\{x,y\})+1$.

If, on the other hand, $\P$ has height 1 and at least three elements, then $\P$ is an antichain which has both dimension and local dimension equal to two, so the Removable Pair Conjecture holds for $\P$.
\end{proof}

Now we return to Theorem~\ref{thm:FourElementsRemoval} to show that, for any poset $\P$ with at least 5 elements, there are four elements that can be deleted from $\P$ such that the local dimension is reduced by at most two.

\begin{proof}[Proof of Theorem~\ref{thm:FourElementsRemoval}]
If $\P$ has height at least 4, then $\P$ has a chain with $4$ elements and, by Corollary~\ref{cor:oneChainRemoval}, the removal of those elements will reduce the local dimension by at most two.
If $\P$ has height at most 2, then we may use Theorem~\ref{thm:twoElementsHeightTwo} twice to find two pairs of elements which will reduce the local dimension by at most two.

So we are left with the case that $\P$ has height exactly $3$.
Consider a 3-element chain $a < b < c$.
If there exists a fourth element $z$ that is incomparable to both $a$ and $c$ (and therefore also $b$) then we can use Theorem~\ref{thm:twoChainsRemoval} to remove the two chains $\{a, b, c\}$ and $z$, reducing the local dimension by at most two.

Now, assume that for each 3-element chain $a<b<c$ in $P$ and each $z$ in $P$, either $a\leq z$ or $z\leq c$. In this case, fix a three element chain $a_0<b_0<c_0$ and observe that $\ldim(\P) \leq \ldim(\P-\{a_0,c_0\})+1$ by Lemma~\ref{lem:twoElementsSpecial}. If $\P-\{a_0,c_0\}$ has height at most 2, then Theorem~\ref{thm:twoElementsHeightTwo} guarantees the existence of two more elements $\{d_0,e_0\}$ such that $\ldim(\P-\{a_0,c_0\}) \leq \ldim(\P-\{a_0,c_0,d_0,e_0\})+1$ as desired.
If instead $\P-\{a_0,c_0\}$ has height 3, then consider a 3-element chain $a_1<b_1<c_1$ in $\P-\{a_0,c_0\}$. Because $\P-\{a_0,c_0\}$ is a subposet of $\P$, this chain was present in $\P$ and so for each element $z$ we have either $a_1\leq z$ or $z\leq c_1$. Therefore, $\ldim(\P-\{a_0,c_0\}) \leq \ldim(\P-\{a_0,c_0,a_1,c_1\})+1$ by Lemma~\ref{lem:twoElementsSpecial} as desired.
\end{proof}

\section{Local difference cover versus local bipartite cover}

Theorem~\ref{thm:CES} and Corollary~\ref{cor:sharp-ex} show that, for $G\sim\mathcal{G}(n/2,n/2,1/e)$, the local complete bipartite cover number ($\lbc(G)$) and the local difference graph cover number ($\ldc(G)$) are bounded below by $(1-o(1))\frac{n}{4e\ln n}$.  Is there a sequence of graphs $(G_n: n\geq 1)$ for which $\ldc(G_n)$ is constant while $\lbc(G_n)$ is unbounded?

Because every nonempty complete bipartite graph is a difference graph, it is clear that $\ldc(G)\leq\lbc(G)$ for every graph $G$.
In Proposition~\ref{prop:ldclbcUB}, we show that, for every difference graph $H=H(m,n;f)$, we have $\lbc(H)\leq\left\lceil \log_2(m+1)\right\rceil$, noting $\ldc(H)=1$. 
As a result, for any graph $G$ with $v$ vertices, $\lbc(G)/\ldc(G) = O(\log v)$. 
\begin{prop}\label{prop:ldclbcUB}
	Let $H=H(m,n;f)$ be a difference graph.
	Then $\lbc(H)\leq\left\lceil \log_2(m+1)\right\rceil$.
	Consequently, for all graphs $G$ on $v$ vertices,
	\begin{align*}
		\ldc(G)\leq\lbc(G)\leq\ldc(G)\left\lceil \log_2(v/2+1)\right\rceil .
	\end{align*}
\end{prop}

\begin{proof}
	We note that it may be convenient to visualize $H=H(m,n;f)$ as a Young diagram in which the $i^{\rm th}$ row has length $f(i)$, for $i\in[m]$, so each square corresponds to an edge in the difference graph. (See Figure~\ref{fig:young}.)
	A complete bipartite graph cover is equivalent to a cover of the Young diagram with generalized rectangles. 
	That is, a bipartite graph corresponds to the product set $S\times T$ so that $S\subseteq[m]$, $T\subseteq[n]$ and $S\times T$ is contained entirely in the Young diagram.
	Then $\lbc(H)$ is the maximum number of generalized rectangles in any row or column.

	We will prove by induction on $m$ that the graph $H$ can be partitioned into complete bipartite graphs so that $\lbc(H)\leq\left\lceil \log_2(m+1)\right\rceil$.
	If $m=1$, the result is trivial. 
	We suppose the statement is true for $m-1\geq 0$ and prove that it is true for $m$.

	We use the complete bipartite graph $\left\{1,\ldots,\lfloor (m+1)/2\rfloor\right\} \times \left\{1,\ldots,f\left(\lfloor (m+1)/2\rfloor \right)\right\}$.
 	Remove the edges in this graph and there are two disconnected components:
	One has dimensions $\lceil (m-1)/2\rceil \times f\left(\lfloor (m+3)/2\rfloor \right)$
	The other has dimensions at most $\lfloor (m-1)/2\rfloor \times \left(n-f\left(\lfloor(m+1)/2\rfloor\right)\right)$.

	By the inductive hypothesis, we can cover the remaining edges of $H$ by complete bipartite graphs (equivalently, cover the remaining blocks of the Young diagram) with each of the components covered by complete bipartite graphs so that each vertex in $H$ appears in at most $\left\lceil\log_2 \left(\lceil (m-1)/2\rceil+1\right)\right\rceil$ of the covering graphs. 
	Hence,
	\begin{align*}
		\lbc(H) 	\leq 	1+\left\lceil\log_2 \left(\lceil (m-1)/2\rceil+1\right)\right\rceil 	= 	\left\lceil\log_2 \left(2\lceil (m+1)/2\rceil\right)\right\rceil . 
	\end{align*}

	This is equal to $\left\lceil\log_2 (m+1)\right\rceil$ if $m$ is odd and is equal to $\left\lceil\log_2 (m+2)\right\rceil$ if $m$ is even.  
	If $m>0$ and $m$ is even, then $\left\lceil\log_2 (m+2)\right\rceil=\left\lceil\log_2 (m+1)\right\rceil$.
\end{proof}

\section{Concluding remarks and open questions}

Theorem~\ref{thm:size-main} establishes that the maximum local dimension of an $n$-element poset is $\Theta(n/\log n)$.
We get a lower bound of $(1-o(1))\frac{n}{4e\ln n}$ and an upper bound of $(1+o(1))\frac{4n}{\log_2 (2n)}$.
We would like to know if the coefficients, $\frac{1}{4e}$ and $4\ln 2$, respectively, could be improved.

In Proposition~\ref{prop:ldclbcUB}, we establish an upper bound on $\lbc(H)$ for difference graphs $H$ that is logarithmic in the smallest partition class, however it is not clear whether this bound is achieved.
We would like to determine the largest value of $\lbc(H)$ over all difference graphs $H$. For difference graph $H_n=H(n,n;f_n)$ with $f_n(i)=n+1-i$, the construction in Figure~\ref{fig:young} for $H_{15}$ can be extended to show $\lbc(H_n)\leq \log(n+1) - 1$ when $n+1$ is a power of 2 and $n\geq 15$, but the following question remains:

\begin{question}\label{conj:ldclbcUB}

                Let $n+1$ be a power of $2$ and let $H_n=H(n,n;f_n)$ be the difference graph such that $f_n(i)=n+1-i$. What is the exact value of $\lbc(H_n)$?

\end{question}

\begin{figure}
		\begin{tikzpicture}[baseline,x=20pt, y=20pt]
			\draw (0,0) rectangle (1,-3);
			\draw (1,0) rectangle (3,-1);
			\draw (1,-1) rectangle (2,-2);
			\foreach \i in {1,...,3}
			{
				\node at (-0.25,-\i+0.5) {\i};	
				\node at (\i-0.5,0.325) {\i};
			}
		\end{tikzpicture}~\hspace*{25pt}~
		\begin{tikzpicture}[baseline,x=10pt, y=10pt]
			\draw (0,0) rectangle (5,-3);
			\draw (5,0) rectangle (7,-1);
			\draw (5,-1) rectangle (6,-2);
			\draw (0,-3) rectangle (3,-5);
			\draw (3,-3) rectangle (4,-4);
			\draw (0,-5) rectangle (1,-7);
			\draw (1,-5) rectangle (2,-6);
			\foreach \i in {1,...,7}
			{
				\node at (-0.5,-\i+0.5) {\i};	
				\node at (\i-0.5,0.625) {\i};
			}
		\end{tikzpicture}~\hspace*{25pt}~
		\begin{tikzpicture}[baseline,x=5pt, y=5pt]
			\draw (0,0) rectangle (8,-8);
			\draw (8,0) rectangle (13,-3);
			\draw (13,0) rectangle (15,-1);
			\draw (13,-1) rectangle (14,-2);
			\draw (8,-3) rectangle (11,-5);
			\draw (11,-3) rectangle (12,-4);
			\draw (8,-5) rectangle (9,-7);
			\draw (9,-5) rectangle (10,-6);
			\draw (0,-8) rectangle (3,-13);
			\draw (0,-13) rectangle (1,-15);
			\draw (1,-13) rectangle (2,-14);
			\draw (3,-8) rectangle (5,-11);
			\draw (3,-11) rectangle (4,-12);
			\draw (5,-8) rectangle (7,-9);
			\draw (5,-9) rectangle (6,-10);
		\end{tikzpicture}
	\caption{Young diagrams are given which represent complete bipartite graph covers  (partitions, in fact) of the edge set of $H_n=H(n,n,f_n)$ with $f_n(i)=n+1-i$ for $n=3,7,15$, respectively. The cases for $n=3,7$ are labeled. The cover for $H_3$ corresponds to the graphs $\{1,2,3\}\times\{1\}$, $\{1\}\times\{2,3\}$, and $\{2\}\times\{2\}$.  The cover for $H_{15}$ shows $\lbc(H_{15}) \leq 3$.}
	\label{fig:young}
\end{figure}
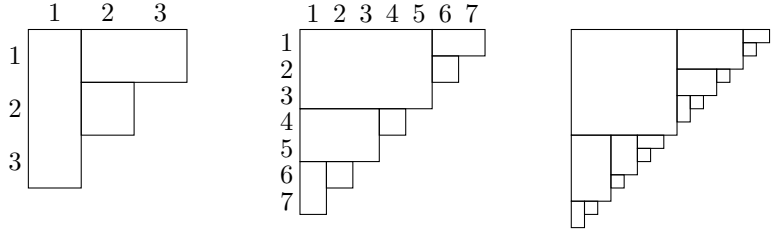

%
%

In Theorem~\ref{thm:BooleanLatticeLB}, it is natural to ask whether the trivial upper bound on $\ldim(\textbf{2}^{n})$ is tight.
\begin{question}\label{conj:BooleanLattice}
    Is it true that $\ldim(\textbf{2}^{n}) = n$ for all $n\geq 1$?
\end{question}

The Removable Pair conjecture is still open for the Dushnik--Miller dimension.
The version for local dimension, Conjecture~\ref{conj:TwoElementsRemoval}, is open as well.

Christophe Crespelle~\cite{CC} observed that, from the information theory perspective, our results imply that local dimension is optimal up to constant factor.
In particular, for a poset $\P$ with $n$ elements, we use $\log{n}$ bits to encode each element. Since the local dimension of $\P$ is $O(n/\log{n})$, $\P$ has a local realizer $\mathcal{L}$ whose ple's use a total of $O(n^2/\log n)$ elements and thus the number of bits used to express $\mathcal{L}$ is $O(n^2)$. This is best possible up to a constant factor because the number of labeled posets with $n$ elements is $2^{\Theta(n^2)}$~\cite{NumberOfPosets}, which means means that $\Theta(n^2)$ bits are needed to encode them as distinct posets.
Note that Dushnik--Miller dimension is not optimal in this respect since there are posets with $n$ elements and dimension $n/2$. So, the linear extensions  in an optimal realizer contain $n^2/2$ total elements which requires $\Theta(n^2\log n)$ bits for encoding.



\section{Acknowledgements}
All authors were supported in part by NSF-DMS grant \#1604458, ``The Rocky
Mountain-Great Plains Graduate Research Workshops in Combinatorics.'' Martin was supported by a grant from the Simons Foundation (\#353292, Ryan R. Martin). Masa\v{r}\'ik was supported by the project CE-ITI P202/12/G061 of GA \v{C}R and by the project SVV-2017-260452 of Charles University. Smith was supported in part by NSF-DMS grant \#1344199. Wang was supported in part by NSF-DMS grant \#1600811.

\end{document}